\documentclass[11pt]{amsart}
\usepackage{graphicx}
\usepackage{latexsym}
\usepackage{amsfonts,amsmath,amssymb}
\usepackage{hyperref}
\newtheorem{theorem}{Theorem}[section]
\newtheorem{proposition}[theorem]{Proposition}
\newtheorem{lemma}[theorem]{Lemma}

\usepackage{color}

\def\pr{\textup{ P\/}}
\def\ex{\textup{E\/}}

\def\eps{\varepsilon}
\def\la{\lambda}
\def\a{\alpha}
\def\be{\beta}

\def\part{\partial}
\def\Cal{\mathcal}

\newcommand{\beq}{\begin{equation}}
\newcommand{\eeq}{\end{equation}}

\newtheorem{Theorem}{Theorem}[section]
\newtheorem{Lemma}[Theorem]{Lemma}

\theoremstyle{remark}

\numberwithin{equation}{section}
\def\bea{\begin{eqnarray}}
\def\eea{\end{eqnarray}}

\date{\today}
\begin{document}

\title[Sparse solutions] {On sparsity of the solution to a random quadratic optimization problem}

\maketitle
\authors{\begin{center}Xin Chen\footnote{Department of Industrial and Enterprise System Engineering, University of Illinois at Urbana-Champaign. Urbana, IL, 61801.  Email: xinchen@illinois.edu.} and Boris Pittel\footnote{Department of Mathematics, Ohio State University, Columbus, OH 43210. Email: bgp@math.ohio-state.edu.



}\end{center}}

\begin{abstract}
The standard quadratic optimization problem (StQP), i.e. the problem of minimizing
a quadratic form $\bold x^TQ\bold x$ on the standard simplex $\{\bold x\ge\bold 0: \bold x^T\bold e=1\}$, is studied. The StQP arises in numerous applications, and it is known to be NP-hard. The first author, Peng and Zhang~\cite{int:Peng-StQP}  showed that almost certainly the StQP with a large random matrix $Q=Q^T$, whose upper-triangular entries are
 i. i. concave-distributed, attains its minimum at a point with few positive components.
In this paper we establish sparsity of the solution for a considerably broader class of the distributions, 
including those supported by $(-\infty,\infty)$, provided that the distribution tail is (super/sub)-exponentially narrow, and also
for the matrices  $Q=(M+M^T)/2$, when $M$ is not symmetric. { The likely support size in those cases is shown to be polylogarithmic in $n$, the problem dimension.} Following~\cite{int:Peng-StQP} and Chen and Peng
~\cite{ChenPeng2015}, the key ingredients are the first and second order optimality conditions, and the integral bound for
the tail distribution of the solution support size. To make these results work for our goal, we obtain
a series of estimates involving, in particular, the random interval partitions induced by the order statistics of the elements $Q_{i,j}$.
\end{abstract}

\section{Introduction and main results} 
Bomze \cite{int:Bomze1} coined the term ``standard quadratic optimization problem'' (StQP) for the problem
\begin{align} \
&\min \bold x^TQ \bold x, \label{StQP} \\
&\text{ s.t. }\bold e^T\bold x=1, \quad \bold x\ge \bold 0\label{simplex},
\end{align}
where $Q=[Q_{ij}]\in \Re^{n\times n}$ is a symmetric matrix, and $\bold e\in \Re^n$ is the all $1$-vector. We will refer
to the set in \eqref{simplex} as the simplex $\Delta_n$.

The StQP appears in numerous applications such as
resource allocation~\cite{Ibaraki}, portfolio
selection~\cite{portfolioselect}, machine learning~\cite{Peng:MKL}, the maximal clique problem in discrete optimization
\cite{int:Gibbons}, and the determination of co-positivity of a matrix in
linear algebra~\cite{int:Bomze3}, etc. Since it is prototype
for numerous, more general, quadratic programming problems, it has been used to test various algorithms proposed in the literature (see \cite{int:Bomze3,int:Scozzari,
int:Yang} and the references therein for details).

Our subject in this paper is a random instance of the StQP,  where the symmetric matrix $Q$ is generated from a certain distribution. To put our work into perspective, { we note that} the study of optimization problems with random data can be traced back to early 1980s, e.g. Goldberg and  Marchetti-Spaccamela~\cite{int:Goldberg} (knapsack problem).  See Beier~\cite{int:Beier} for a more recent progress on random knapsack problems.  There has been 
made a significant progress in analysis of the so-called $L_1$ minimization problem with random constraints.
Notably it was proved that
when the coefficient matrix is generated from a normal distribution, then with high probability (whp), the optimal solution of the $L_1$ minimization problem is the sparsest point in the constrained set (see Cand\'es and Wakin~\cite{int:Candes1}, Cand\'es, Romberg and Tao~\cite{int:Candes2}, and Donoho~\cite{int:Donoho1}). 


{It is also important to note that in the optimization literature, when testing algorithms, it is not uncommon to generate optimization problem data randomly due to the lack of testing instances.  For example, Bomze and De Klerk \cite{int:Bomze02} and Bundfuss and D\"{u}r \cite{BundfussDur2009} generate StQPs with symmetric $Q$ whose entries are uniformly distributed. Thus, a good understanding of the behavior of the optimal solutions under randomly generated instances may shed light on the behaviors of various algorithms tested on these instances. Indeed, our results, together with those in \cite{int:Peng-StQP} and \cite{ChenPeng2015}, establishing the sparsity of the optimal solutions of randomly generated StQPs under quite general distribution assumptions, indicate that the performance of algorithms tested on these instances must be carefully analyzed before any general statement can be made. Interestingly, motivated by the sparsity of the optimal solutions, Bomze et al.\cite{BomzeSchachingerUllrich2017} construct StQP instances with a rich solution structure.  }

The first author, Peng and Zhang~\cite{int:Peng-StQP}, { prodded by a close relation between the StQP and the $L_1$ minimization problem and a keen interest in understanding randomly generated optimization problems,} proved that, as $n\to\infty$,  the random StQP whp has an optimal solution $\bold X^*$ with the number of non-zero components bounded in probability, provided
that the distribution $F$ of $Q_{i,j}$, ($i\le j$), is supported by $[A,B)$, with finite $A$, and $F$ is concave on $[A,B)$.
This family of distributions contains, for instance, the uniform distribution and the exponential distribution. However,
the concavity assumption excludes $A=-\infty$, whence the normal distribution was out.  In a follow-up 
work, Chen and Peng~\cite{ChenPeng2015} were still able to prove that for the GOE matrix $Q=(M+M^T)/2$, $M_{i,j}$ being i.i.d. normal, whp the minimum point $\bold X^*$ has at most {\it two\/} non-zero components, thus being almost 
an extreme vertex of the simplex. The key ingredient of the proof was an upper bound $e^{-n^2/4}$ for the probability that $n$-dimensional GOE matrix is positive, semi-definite (see Dean and Majumdar~\cite{int:PSD-GOE}).

The core analysis in \cite{int:Peng-StQP} is based on the estimates coming from the first-order optimality condition 
on $\bold X^*$ and some probabillistic bounds on the order statistics for the attendant random variables.
The further analysis in \cite{ChenPeng2015} relied {in addition} on the second-order optimality condition. 

These two optimality conditions, and a pair of the integral inequalities from \cite{int:Peng-StQP}, remain our tools in this paper as well.  Still we have to extend, considerably, the
probabilistic/combinatorial techniques to broaden the class of admissible distributions beyond concave ones, in order
to include the distributions with support going all the way to $-\infty$.

Our main results are as follows. Suppose that the cumulative distribution function (cdf) $F$ of the entries $Q_{i,j}$, $i\le j$, is
continuous. Let $\bold X^*$ denote a global minimum solution of the StQP with the random matrix $Q$. Let $K_n$
denote the support size for $\bold X^*$, i.e. $K_n=|\{j\in [n]:\,X_j^*>0\}|$.

\begin{Theorem}\label{A} {Let $\alpha>e\sqrt{2}$, and $k_n=\lceil \a n^{1/2}\rceil$. Then }
\[
\pr\{K_n\ge k_n\} =O\bigl(e^{\gamma(\alpha)n^{1/2}}\bigr), \quad \gamma(\a):=2\a\log(e\sqrt{2}/\a) < 0.
\]
So $K_n=O(n^{1/2})$ with probability sub-exponentially close to $1$.
\end{Theorem}

The surprisingly short proof of this general claim follows from Theorem 3 in ~\cite{ChenPeng2015}. With additional, mildly
restrictive, constraints, we are able to reduce, significantly, the likely range of the support size.

\begin{Theorem}\label{B} (Left-bounded support.) Suppose that the c.d.f. $F(x)$ has a continuous density $f(x)$, and satisfies the following properties. 
\begin{itemize}
\item[(1)] { $f(x)>0$ for $x\in [A, B)$, $-\infty<A<B\le\infty$; $A=0$} without loss of generality; 
\item[(2)] There exists $\nu>0$ and $\rho>0$ such that
\begin{equation*}
F(x)=\rho x^{\nu} +O(x^{\nu+1}),\quad x\downarrow 0;
\end{equation*}
\item[(3)]
\begin{equation*}
\beta:=\sup\left\{\frac{f(x')}{f(x)}:\,x,x'\in (0,B),\,x'\in (x,2x)\right\}<\infty.
\end{equation*}
\end{itemize}
Then, for $k\le k_n$,
\begin{align*}
\pr\{K_n=k\}&\le \exp\bigl(-c(\max(\nu,1)) k+o(k)\bigr),\\
c(\mu)&:=\int_0^1\!\log(1+x^{\mu})\,dx=\sum_{j\ge 1}\frac{(-1)^{j-1}}{j(j\mu+1)}.
\end{align*}
\end{Theorem}
\noindent {\bf Notes.\/} {\bf (i)\/} So $\pr\{K_n>k\}$ decays exponentially fast long before $k$ gets as large as $k_n$. 
{In fact, the exponential decay holds much longer, up until $k$ is, say, of order $n\log^{-2} n$.} That is, in probability, $K_n$ is very much bounded. 
{\bf (ii)\/} The uniform distribution and the exponential distribution  are covered by this theorem: $\nu=1$, $\beta=1$ for the former, and $\nu=2$, $\beta=1$ for the latter.  Notice also that the leading term $\rho x^{\nu}$ in the condition (2) is {concave for $\nu\le 1$}. The local nature
of this condition makes Theorem \ref{B} substantially more general than Theorem 3.4 in \cite{int:Peng-StQP}, proved
for $F(x)$ concave everywhere on the support of the density $f$, whence for the two distributions mentioned above.
{\bf (iii) \/} The condition (3) is an extension of the notion of a ``dominatedly varying'' {\it monotone\/} function introduced and
studied by Feller \cite{Fel1}, \cite{Fel2} (Ch. 8, Exer. 33).

While the class of distributions covered by Theorem \ref{B} is quite broad, it does not contain, for instance, the
normal distribution (supported by $(-\infty,\infty)$), which is predominantly assumed in the extensive literature on the $L_1$ minimization problem and its various generalizations. 
With the normal distribution in our peripheral vision, we introduce a class of distributions supported { by $(-\infty,B)$, $B\le\infty$, such that, for $x\to-\infty$,
\begin{equation}\label{F"norm"}
F(x)= (c+O(|x|^{-\kappa})) |x|^a \exp(-r |x-x_0|^b),\quad b,c,r,\kappa>0;
\end{equation}
shifting/scaling $x$ we make $r=1$, $x_0=0$.} It is well-known that the  normal distribution meets \eqref{F"norm"} with  
$a=-1$, $b=2$. Among other examples are the two-sided exponential density ($a=0$, $b=1$),
and the $\cosh$-density ($a=0$, $b=1$). As in Theorem \ref{B}, this condition restricts the behavior of the cdf
$F(x)$ in the vicinity of a {\it single\/} point, which is $-\infty$ this time. 
{
\begin{Theorem}\label{C} Let $b>1$. 
Then,  for all $k\le k_n$, we have
\[
\pr\bigl\{K_n\ge k\bigr\}\le_b n \left(\frac{8}{9}\right)^{k/4}+n\exp\Biggl(-\frac{k\left(\log\frac{n}{k}\right)^{\min\{0,a/b\}}}{2e}\Biggr);
\]
the symbol $\le_b$ means that the probability is below the RHS times a bounded coefficient.
Consequently $K_n=O_p(\log n)$ for $a\ge 0$, meaning that $\pr(K_n > \omega(n)\log n)\to 0$ for
{\bf every\/}  $\omega(n)\to\infty$ {\bf however
slowly\/}), and $K_n=O_p\bigl((\log n)^{1+|a|/b}\bigr)$ for $a<0$.
\end{Theorem}

\begin{Theorem}\label{D} Let $b\le 1$. Define
\[
\sigma(a,b)=\left\{\begin{aligned}
&1+\frac{1+2a}{b},&&\text{ if }a>0,\\
&1+\frac{1+|a|}{b},&&\text{ if }a\le 0;
\end{aligned}\right. 
\]
so $\sigma(a,b)>2$.  For every $\sigma>\sigma(a,b)$, and $d<\frac{b(\sigma-\sigma(a,b))}{2}$,
\[
\pr\{K_n > \log^{1+d} n\}< \exp\bigl(-\log^{1+d} n\bigr).
\]
\end{Theorem}
}
{\bf Note.\/} Therefore, with probability $> 1 - O(n^{-L})$, ($\forall\,L>0$),  $K_n$ 
is below $\log^{1+d} n$. Using the term coined in Knuth, Motwani and Pittel~\cite{KnuMotPit}, ``quite
surely'' (q.s.) the support size is of a moderate (poly-logarithmic) order. 

Turn to the alternative model: first randomly generate an $n\times n$ matrix $M$ whose elements are i.i.d.
random variables with the cdf $G$; then define $Q=(M+M^T)/2$. 

Suppose $G$ satisfies
the (one-point) condition \eqref{F"norm"}, and $x_0=0$, $r=1$, without loss of generality.  The diagonal entries of $Q$ have distribution $G$ and the non-diagonal
entries of $Q$ have the distribution $F(x)= (G\star G)(2x)$, $\star$ standing for convolution. We prove that, for $a>-1$ when
$b\le 1$,  $F$ satisfies the equation \eqref{F"norm"}
as well, {with the parameters $b'=b$, $c'>0$, $r'=2^{\text{min}(1,b)}$ and
\begin{equation*}
a'=\left\{\begin{aligned}
&2a+\frac{b}{2},&&\text{ if }\, b>1,\\
&a+b-1,&&\text{ if }\,0<b<1,\\
&2a+1,&&\text{ if }\,b=1.
\end{aligned}\right.
\end{equation*}
Since $r'>1=r$, }we have $\lim_{x\to -\infty}G(x)/F(x)=
\infty$. Combining this fact and the general identity proved in ~\cite{int:Peng-StQP} (the proof of Theorem 2.2), we easily transfer the proof
of our Theorems \ref{C} and \ref{D} to this model. To state the resulting claims one has only to
replace $a$ with $a'$ in Theorems \ref{C} and \ref{D}. 

We note that the theorems above have the natural counterparts for the problem $\max\{\bold x^TQ\bold x:\,\bold x\text{ meets }
\eqref{simplex}\}$: the distribution $F$ of $Q_{i,j}$ is supported by $(A,\infty)$, $A\ge -\infty$, and the restrictions are imposed on the behavior of $F(x)$ in the vicinity of $\infty$. Since $-Q$ meets the conditions of the respective claim {for the
support $(-\infty,-A)$,} no additional proof is needed.
{So, for the {\it quasi\/}-normal distribution, i.e. $a=-1$, $b=2$, both the minimum point and the maximum point are sparse, with the likely support size of order $(\log n)^{3/2}$ in each case. As we mentioned, Chen and Peng~\cite{ChenPeng2015} proved  that for the GOE matrix $Q=(M+M^T)/2$, $M_{i,j}$ being i.i.d. {\it exactly\/} normal, whp the support size of $\bold X^*$  is $2$, at most.}

{
To conclude the introduction, we mention that thirty years earlier Kingman \cite{Kin} initiated the study of {\it local\/} maxima of
the random quadratic forms $\bold p^T F\bold p$ on the simplex $\Delta_n$, with $\bold p$ interpreted as the distribution of 
the alleles $A_1,\dots, A_n$ at a single locus, and $F_{i,j}\in [0,1]$ as the (random) fitness, i.e.  the probability that the gene pair $(A_i,A_j)$ survives to a reproductive age.  Kingman's main interest was potential coexistence of several alleles at the locus, i.e. of locally stable
distributions (local maxima) $\bold p$ with a sizable support. Concurrently, in the context of evolutionary stable strategies, Haigh \cite{Hai1}, \cite{Hai2} established
the counterparts of some of Kingman's results for a non-symmetric payoff matrix; see a more recent paper
Kontogiannis and Spirakis \cite{KonSpi}. The second author of the current paper showed that for a broad class of  
the fitness distributions, the likely support of a local maximum point $\bold p$ not containing a support of a local {\it equilibrium\/} is $\lceil (2/3)\log_2n\rceil$, at most. And, for the uniform fitnesses, there are likely many {\it potential\/} local maxima supports free of local equilibriums, each of size close to $\lceil(1/2)\log_2n\rceil$, \cite{Pit}. Conjecturally the likely size of the largest support of a local maximum is polylogarithmic in $n$.}


{The paper is organized as follows. In Section \ref{sec:preliminaries}, we provide some preliminaries useful for our analysis. In Section \ref{sec:proofs}, we present the proofs of our key results, {and finish the paper with some concluding remarks in }Section \ref{sec:conclusion}. }



\section{Preliminaries}\label{sec:preliminaries}
The analysis of the random StQP in \cite{int:Peng-StQP} began with the formulation and the proof of the following optimality
conditions.
 Given $\bold x\in \Delta_n$, denote $k(\bold x)=|\{j\in [n]: x_j>0\}|$.
\begin{proposition}
\label{thm:Optimality}
	Suppose that $\bold x^*$ is an optimal solution of the problem (\ref{StQP}-\ref{simplex}) satisfying $k(\bold x^*)=k>1$.
	So $\lambda^*:=(\bold x^*)^T Q \bold x^*$ is the absolute minimum of the quadratic form on the simplex $\Delta_n$.
	Denoting $\Cal K=\{j\in [n]: x^*_j>0\}$, let $Q_{\Cal{K}}$
	be the principal $k\times k$ submatrix of $Q$ induced by the elements of the set $\Cal K$. Then
	\begin{itemize}
		\item[C.1]  there exists a row (whence a column) of $Q_{\Cal{K}}$ such that the arithmetic mean of all its elements is 
		(strictly) less than $\min_{j\in [n]} Q_{j,j}$;
		\item[C.2] with $E_{\Cal K}(i,j):=1_{\{i,j\in \Cal K\}}$, $Q_{\Cal{K}}-\lambda^* E_{\Cal K}$ is positive               semidefinite; in short, $Q_{\Cal{K}}-\lambda^* E_{\Cal K}\succcurlyeq \bold 0$.
	\end{itemize}
\end{proposition}

Properties C.1 and C.2 follow, respectively, from the first-order optimality condition and the second-order optimality condition.

Consider the random symmetric matrix $\{Q_{i,j}\}$: (1) the diagonal entries $Q_{i,i}$ are independent, each having the same, continuous, distribution $G$; (2) the above-diagonal elements $Q_{i,j}$, $(i<j)$, are independent of each other, and of the diagonal elements, each having the same, continuous, distribution $F$; (3) the below-diagonal elements $Q_{i,j}$ are set equal  to $Q_{j,i}$. 

If we relabel the elements of $[n]$ to make $Q_{1,1}<Q_{2,2}<\cdots<Q_{n,n}$, then the above-diagonal elements in the new $n\times n$ array will remain independent of each other and of the diagonal entries, that now form the sequence $V_1,\dots, V_n$ of the {\it order statistics\/} for $n$ independent variables, each with distribution $G$.  

Let us use the capital $\bold X^*$ to denote the solution of the random StQP problem. Let $K_n$ denote the support size of  $\bold X^*$.
Property C.1 was used in \cite{int:Peng-StQP} to prove the following, crucial, estimate.
\begin{Lemma}\label{bound1} Let $V_1<V_2<\cdots< V_n$ ($W_1<W_2<\cdots<W_{n-1}$ resp.) denote the order statistics of the sequence of $n$ ($n-1$ resp.) independent, $G$-distributed ($F$-distributed resp.) random variables. Assume that $\bold V:=(V_1,\dots,V_n)$
and $\bold W:=(W_1,\dots,W_{n-1})$ are independent of each other. Then, for $k\ge 1$,
\begin{align*}
&\qquad\qquad \pr\{K_n=k+1\}\le \rho(n,k),\\
&\rho(n,k):=\sum_{i=1}^n\!\pr \bigl\{\bar W_k\le (k+1)V_1-V_i\bigr\},\quad \bar W_k:=\sum_{j=1}^{k} W_j.
\end{align*}
\end{Lemma}
\begin{proof} (For completeness.) Consider $Q$ obtained from the initial $\{Q_{i,j}\}$ via the above relabeling, so that now $Q_{i,i}$ is the $i$-th
smallest among the diagonal entries. For each $i\in [n]$, let $\bold W(i)=(W_1(i)<\cdots< W_{n-1}(i))$ stand for the $(n-1)$ order statistics of
the non-diagonal  entries in the $i$-th row of the transformed $Q$. $\bold W(1),\dots, \bold W(n)$ are equi-distributed, 
independent of $\bold V$, {\it though not of each other\/} since the matrix is symmetric. By the property C.1, there exists a row $i^*$ in $Q$ such the sum of some $k+1$ entries in this row, that
includes its diagonal entry, is below $(k+1)\min_i Q_{i,i}=(k+1)V_1$. This sum is certainly not smaller than
\[
\sum_{j=1}^k W_j(i^*)+V_{i^*}=\bar W_k(i^*)+V_{i^*}. 
\]
For a generic row $i\in [n]$,
\[
\pr\bigl\{\bar W_k(i)+V_i\le (k+1)V_1\}=\pr\bigl\{\bar W_k+V_i\le (k+1)V_1\}.
\]
Applying the union bound we complete the proof.
\end{proof}
In the case $k=n$ Property C.2 was utilized in \cite{ChenPeng2015} to prove 
\begin{Lemma}\label{bound2} For $n\ge 2$,
\[
\pr\left\{K_n=n\right\}\le\pr\left\{\bigcap_{i\neq j\in [n]}\bigl\{Q_{i,j}\le \max(Q_{i,i},Q_{j,j})\bigr\}\right\}\le \frac{2^n}{(n+1)!}.
\]
\end{Lemma}

Define 
\[
\rho(n,k)=\pr\bigl\{\bar W_k\le kV_1\}+\hat{\rho}(n,k), 
\]
\[
\hat{\rho}(n,k):=\sum_{i=2}^n\pr\{\bar W_k\le (k+1)V_1-V_i\}.
\]
{Notice that $\hat{\rho}(n,k)\le (n-1)\pr\bigl\{\bar W_k\le kV_1\}$, since $V_i\ge V_1$ for $i\ge 2$. Therefore, by Lemma 
\ref{bound1},
\begin{equation}\label{P(Kn=k+1)<rho(n,k)}
\pr\{K_n=k+1\}\le n \pr\bigl\{\bar W_k\le kV_1\}.
\end{equation}
}
Using the classic formula for the joint distribution of the order statistics, they found in \cite{int:Peng-StQP} (see (11), and
 the top identity in the proof of Theorem 2.2 therein) the multi-dimensional integral representations of the functions $\rho(\cdot)$ and $\hat\rho(\cdot)$. 
 
In terms of the order statistics $\bold W$, and $\bar W_k:=\sum_{j\in [k]}W_j$, the formulas for $\pr\bigl\{\bar W_k\le kV_1\}$
and $\hat\rho(n,k)$ become  
 \begin{equation}\label{P(barW<),hatrho=}
 \begin{aligned}
 \pr\bigl\{\bar W_k\le kV_1\}&=\ex\!\left[\bigl(1-G(\bar W_k/k)\bigr)^n\right],\\
 \hat\rho(n,k)&=(n-1)\,\ex\!\left[H_{n,k}(\bar W_k)\right],\\
 H_{n,k}(w)&\!:=-\frac{[1-G(w/k)]^n}{n-1}\\
&\quad\, +\frac{n(k+1)}{n-1}\!\int\limits_{w/k}^{\infty}\!g\bigl((k+1)v-w)\,[1-G(v)]^{n-1}\,dv.
 \end{aligned}
 \end{equation}
The top formula was the main focus of analysis in \cite{int:Peng-StQP}, and will be instrumental in our paper as well. 
As in \cite{int:Peng-StQP}, we switch to $U_j=F(W_j)$, so that $U_j$ are order statistics for the variables $F(X_j)$, where
$X_j$ are i.i.d. with the cdf $F$. We know, of course, that $F(X_j)$ are $[0,1]$-uniform. Thus $U_1,\dots,U_{n-1}$ are
the order statistics of the sequence of $(n-1)$ independent, uniformly distributed, random variables. So the formula 
of the keen interest becomes
\begin{equation}\label{key}
\pr\bigl\{\bar W_k\le kV_1\}=\ex\Biggl[\!\Biggl(1-G\Biggl(\frac{1}{k}\sum_{j=1}^kF^{-1}(U_j)\Biggr)\!\Biggr)^{\!n}\Biggr].
\end{equation}


\section{Proofs}\label{sec:proofs}
For convenience, as we go along, we will restate the claims already made in Section 1.

{Let $G=F$.} We begin with Theorem \ref{A}, a really low hanging fruit, { that nevertheless will significantly influence our estimates throughout the paper.} 
\begin{Theorem}\label{A1} Let $K_n$ be the support size for the solution of the random StQP with continuously
distributed entries. Picking $\alpha>e\sqrt{2}$ and setting $k_n:=\lceil\alpha n^{1/2}\rceil$, we have
\[
\pr\{K_n \ge k_n\} =O\bigl(e^{\gamma(\alpha)n^{1/2}}\bigr), \quad \gamma(\a):=2\a\log(e\sqrt{2}/\a) <0.
\]
\end{Theorem}
\begin{proof} From Proposition \ref{thm:Optimality} C.2 and Lemma \ref{bound2}, we have that
\begin{align*}
\pr\{K_n=k\}& \le  \pr\Bigl\{\exists\text{ a }k\times k\text{ submatrix } {\mathcal K} 
\mbox{ s.t. } Q_{\mathcal{K}}-\lambda^*E_k\succcurlyeq \bold 0\Bigr\} \\
& \le S(n,k):=\binom{n}{k}\cdot\frac{2^k}{(k+1)!}.
\end{align*}
{ Using the inequality $b!\ge (b/e)^b$ (implied by $(1+1/b)^b<e$) and its corollary $\binom{a}{b}\le (ea/b)^b$, }we obtain
\[
S(n,k)\le \left(\frac{2e^2n}{k^2}\right)^k\le \left(\frac{2e^2}{\alpha^2}\right)^k,\quad \forall\, k\ge \alpha n^{1/2}.
\]
Summing up this bound for $k\ge k_n=\lceil \alpha n^{1/2}\rceil$, we complete the proof.
\end{proof}



Turn to the integral formula \eqref{key}.  Below and elsewhere we will write $L\le_b R$ when $L=O(R)$, but $R$ is bulky
enough to make $O(R)$ unsightly.

{First, observe that, given $\delta\in (0,1)$, by the union bound we have: for $k\le k_n$, and $n$ large,
\begin{equation}\label{Uk<delta}
\begin{aligned}
\pr\{U_k\ge \delta\}&\le \binom{n-1}{n-k}(1-\delta)^{n-k}\le\binom{n}{k}(1-\delta)^{n-k}\\
&\le (1-\delta)^{n-k} n^k\le_b (1-\delta)^{0.99n} n^k
\le n^{-k},
\end{aligned}
\end{equation}
provided that $\delta=\delta_n:=2.1n^{-1}k_n\log n=O(n^{-1/2}\log n)$. So the contribution to the RHS of \eqref{key} coming from $U_k\ge \delta_n$ is at most $n^{-k}$, uniformly for $k\le k_n$.}

Second, let us show that we can focus on $\bold U$ with 
\begin{equation}\label{second}
S(\bold U):=\frac{1}{k}\sum_{j=1}^k \log\frac{1}{U_j}\lesssim \log\frac{ne}{k}.
\end{equation}
This fact will come handy in the analysis of the $F$'s support going all the way to $-\infty$.

{
\begin{Lemma}\label{A1<S<A2} Given $\a>0$, define $S=S(\a)= \log\frac{n}{\a k}$. If $\a<e^{-1}$, then
for every $\beta \in (0, 1-\a e)$ we have
\[
\pr\bigl\{S(\bold U)>S\bigr\}
\le_b\Biggl(\frac{\a e}{1-\be}\Biggr)^{\be k}.
\]
\end{Lemma}

\begin{proof}  Recall that the components of $\bold U$ are the $k$ first order statistics of $(n -1)$ independent, $[0,1]$-uniform random variables. In view of the sum-type formula for $S(\bold U)$ in \eqref{second}, we apply a variation of
Chernoff' method. Picking $\la>0$, and $\la<k$ in order to make the coming integral finite, we define $\bold u=(u_1<\cdots
<u_k)\in (0,1)^{n-1}$ and write
\begin{align*}
\pr\bigl\{S(\bold U)&>S\bigr\}\le e^{-\la S} (n-1)_k\!\!\int\limits_{S(\bold u)>S}\!\!\!\!\!\!(1-u_k)^{n-1-k}
\exp\!\left(\!\frac{\la}{k}\sum_{j=1}^k\log\frac{1}{u_j}\!\right)
d\bold u\\
&=\frac{e^{-\la S}(n-1)_k}{(k-1)!}\int_0^1 (1-u_k)^{n-1-k} u_k^{-\la/k}\left(\int_0^{u_k} w^{-\la/k}\,dw\right)^{k-1}\!\!\!du_k\\
&=\frac{e^{-\la S}(n-1)_k}{(1-\la/k)^{k-1}\,(k-1)!}\int_0^1 (1-u_k)^{n-1-k} u_k^{k-1-\la}\,du_k\\
&=\frac{e^{-\la S}(n-1)_k}{(1-\la/k)^{k-1}\,(k-1)!}\cdot\frac{\Gamma(n-k)\,\Gamma(k-\la)}{\Gamma(n-\la)}\\
&=\frac{e^{-\la S}}{(1-\la/k)^{k-1}}\cdot\frac{\Gamma(n)\,\Gamma(k-\la)}{\Gamma(n-\la)\,\Gamma(k)}.
\end{align*}
Set $\lambda=\beta k$;  using Stirling formula for the Gamma function in the
last expression, it is easy to see that
\[
\pr\bigl\{S(\bold U)>S\bigr\}\le_b\left(\frac{\alpha e}{1-\beta}\right)^{\beta k};
\]
the bound is exponentially small since $\alpha e<1-\beta$. 
\end{proof}
}

\subsection{Distributions with left-bounded supports}

In this section, we focus on a class of distributions satisfying the properties in Theorem \ref{B}. 

In \eqref{Uk<delta} we showed that, at the cost of $O(n^{-k})$ error term, we can neglect $\bold U$ with $U_k>
\delta_n =2.1n^{-1}k_n\log n$, for $k\le k_n=\lceil \alpha n^{1/2}\rceil$. We need to show that, for the remaining $\bold U$, 
$\phi(\bold U):=F\bigl(k^{-1}\sum_{j=1}^k F^{-1}(U_j)\bigr)$ typically dwarfs $1/n$ for large $k$,  and so makes $(1-\phi(\bold U))^n$ 
close to zero in all likelihood.
Our first step is to establish an explicit lower bound for $\phi(\bold u)$.
\begin{Lemma}\label{F(F^{-1}/k)} Assume that $F$ meets the conditions (1) and (2) in Theorem \ref{B}. 

{\bf (1)\/} There exists $\gamma=1+O\bigl(\delta_n^{1/\nu}\bigr)$
such that, uniformly for $\bold u=0< u_1\le\cdots\le u_k \le \delta_n$, we have: $\phi(\bold u)\ge \gamma\Bigl(k^{-1}\sum_{j =1}^ku_j^{1/\nu}\Bigr)^{\nu}$.

{\bf (2)\/} Further, for $\nu\ge 1$, we have $\phi(\bold u)\ge \sum_{j=1}^k \gamma_j u_j$,  with
\begin{equation}\label{gamma_j}
\begin{aligned}
&\gamma_j:=\gamma\left[\left(1-\frac{j-1}{k}\right)^{\nu}-\left(1-\frac{j}{k}\right)^{\nu}\right].\\
\end{aligned}
\end{equation}
\end{Lemma}

\begin{proof} {\bf (1)\/} {We will use the notation $g=\Theta(f)$, if $g,f>0$ and $g\ge c f$ for an absolute constant $c>0$ in
a given range of the arguments of $f$ and $g$. }

Since $F(x)=\Theta(x^{\nu})$, $(x\downarrow 0)$, we see that $F^{-1}(u)\in [0,F^{-1}(\delta_n)]$ for $u\in [0,\delta_n]$, and $F^{-1}(\delta_n)=O\bigl(\delta_n^{1/\nu}\bigr)$.   Define
\[
\sigma=\min_{x\in [0,F^{-1}(\delta_n)]} \frac{F(x)}{x^{\nu}}, \quad \eta=\frac{1}{\max_{x\in [0,F^{-1}(\delta_n)]} 
\frac{F(x)}{x^{\nu}}},\quad \gamma=\frac{\sigma}{\eta}.
\]
Since $F(x) =\rho x^{\nu} +O(x^{\nu+1})$, we have $\sigma, \eta=\rho+O\bigl(\delta_n^{1/\nu}\bigr)$,
so $\gamma=1+O\bigl(\delta_n^{1/\nu}\bigr)$, and furthermore
\begin{equation}\label{F,F^{-1}>}
F(x)\ge \sigma x^{\nu},\,\,\forall\, x\in \bigl[0,F^{-1}(\delta_n)\bigr];\quad
F^{-1}(u)\ge \left(\frac{u}{\eta}\right)^{1/\nu},\,\,\forall\,u\in [0,\delta_n].
\end{equation}
 Applying \eqref{F,F^{-1}>},  we lower-bound
\begin{align*}
\phi(\bold u) &\ge \sigma\Bigl(k^{-1}\sum_{j=1}^kF^{-1}(u_j)\Bigr)^{\nu}
\ge \gamma\Bigl(k^{-1}\sum_{j =1}^ku_j^{1/\nu}\Bigr)^{\nu}.
\end{align*}
{\bf (2)\/}
Suppose $\nu\ge 1$. To explain $\gamma_j$  in \eqref{gamma_j}, let $\nu$ be an integer. Using notation $\binom{\nu}{\nu_1,\dots,\nu_k}$
for the multinomial coefficient $\nu!/[\nu_1!\cdots\nu_k!]$, ($\sum_j \nu_j=\nu$), we have
\begin{multline*}
\gamma\left(\frac{1}{k}\sum_{j=1}^k u_j^{1/\nu}\right)^{\nu}=\frac{\gamma}{k^{\nu}}\sum_{\nu_1,\dots,\nu_k}
\binom{\nu}{\nu_1,\dots,\nu_k}\prod_{j=1}^k u_j^{\nu_j/\nu}\\
\text{partitioning the sum according to the first }j\\
\text{ such that }\nu_j>0\text{ and using }u_j=\min_{\ell\ge j}u_{\ell}\le 1\\
\ge\frac{\gamma}{k^{\nu}}\sum_{j=1}^k u_j\sum_{\nu_j+\cdots+\nu_k=\nu,\,\nu_j>0}\binom{\nu}{\nu_j,\dots,\nu_k}\\
= \frac{\gamma}{k^{\nu}}\sum_{j=1}^k u_j \bigl((k-j+1)^{\nu}-(k-j)^{\nu}\bigr)
=\sum_{j=1}^k \gamma_j u_j.
\end{multline*}
Emboldened by this derivation, let us show that this inequality holds for all $\nu\ge 1$. Clearly our task is to prove that, for 
\begin{equation}\label{1constraint}
0\le v_1\le v_2\le\cdots\le v_k,
\end{equation}
we have
\[
\left(\sum_{j=1}^k v_j\right)^{\nu}\ge \sum_{j=1}^k v_j^{\nu}\bigl[(k-j+1)^{\nu}-(k-j)^{\nu}\bigr].
\]
Without loss of generality, we may assume that $\sum_{j=1}^k v_j=1$. We need to show that, subject to this constraint,
the maximim value of the RHS function, call it $\psi(\bold v)$, is $1$. Since $\nu\ge 1$, $\psi(\bold v)$ is a convex function of $\bold v\ge \bold 0$. So, for $\bold v$ meeting \eqref{1constraint} {\it and\/}
$\sum_{j=1}^k v_j=1$, $\psi(\bold v)$ attains its maximum at an extreme point of the resulting polyhedron.
Every such point $\bold v$ is of the form $\bold v=(0,\dots 0, v,\dots,v)$. So if the last zero is at position $j_0$, then
$(k-j_0)v=1$, i.e. $v=(k-j_0)^{-1}$, and therefore $\psi(\bold v)$ is
\begin{multline*}
\sum_{j=1}^kv_j^{\nu} \bigl[(k-j+1)^{\nu}-(k-j)^{\nu}\bigr]\\
=(k-j_0)^{-\nu}\sum_{j=j_0+1}^k\bigl[(k-j+1)^{\nu}-(k-j)^{\nu}\bigr]\\
=(k-j_0)^{-\nu}\cdot (k-j_0)^{\nu}=1.
\end{multline*}
\end{proof}
Armed with this lemma, we derive the upper bound for the truncated expectation 
\[
E_{n,k}:= \ex\Bigl[1_{\{U_k\le\delta_n\}}\bigl(1-\phi(\bold U)\bigr)^n\Bigr], \quad \phi(\bold U)=F\Bigl(k^{-1}\sum_{j =1}^kF^{-1}(U_j)\Bigr).
\]
Let $\nu\ge 1$.  Using notations $\gamma_{j:k}=\sum_{\ell= j}^{k}\gamma_{\ell}$, 
 $d\bold u_t=\prod_{j=1}^t du_j$, $u_0=0$, we write
\begin{align*}
&\frac{E_{n,k}}{(n-1)_k} = \!\!\int\limits_0^{\delta_n}\cdots\!\!\int\limits_{u_{k-1}}^{\delta_n}\!\!(1-u_k)^{n-1-k}
\bigl(1-\phi(\bold u)\bigr)^n d\bold u_k\\
&\le \!\!\int\limits_0^{\delta_n}\cdots\!\!\int\limits_{u_{k-1}}^{\delta_n}\!\!(1-u_k)^{n-1-k}
\Bigl(1-\sum_{j=1}^k\gamma_ju_j\Bigr)^n d\bold u_k\\
&\qquad\qquad\qquad\qquad \text{using concavity of }\log(\cdot)\\
&\le\!\!\int\limits_0^{\delta_n}\cdots\!\!\int\limits_{u_{k-1}}^{\delta_n}\Biggl[1-u_k+\frac{n}{2n-1-k}\Biggl(u_k(1-\gamma_k)-\sum_{j=1}^{k-1}\gamma_ju_j\Biggr)\Biggr]^{2n-1-k} d\bold u_k\\
&\quad\,\text{integrating by parts over }u_k\text{ and dropping the negative term at } u_k=\delta_n\\
&\le (2n-k)^{-1}\left(1-\frac{n}{2n-1-k}(1-\gamma_k)\right)^{-1}\\
& \times\!\!\int\limits_0^{\delta_n}\cdots\!\!\int\limits_{u_{k-2}}^{\delta_n}
\Biggl[1-u_{k-1}+\frac{n}{2n-k}\Biggl(\!u_{k-1}(1-\gamma_{k-1:k})-\sum_{j=1}^{k-2}\gamma_ju_j\Biggr)\Biggr]^{2n-k+1} \!\!d\bold u_{k-1}.
\end{align*}
Repeating the integration step $(k-1)$ times, we arrive at the bound
\begin{align*}
E_{n,k}&\le 
\prod_{j=1}^k \frac{n-1-k+j}{2n-1-k+j}\prod_{j=1}^k\!\left(1-\frac{n}{2n-1-k}(1-\gamma_{k-j+1:k}) \!\right)^{-1}\\
&=\prod_{j=1}^k \frac{n-1-k+j}{2n-1-k+j}\prod_{j=1}^k\!\left(\!1-\frac{n}{2n-1-k}\left(\!1-\gamma\left(\frac{j}{k}\right)^{\nu}\right)\!\right)^{-1}.\\
\end{align*}

\noindent Taking logarithms and using  $k\le k_n$, $\gamma=1+O(\delta_n^{1/\nu})$, we easily obtain
\begin{equation}\label{easily}
\begin{aligned}
&\qquad\qquad\qquad \frac{\log E_{n,k}}{k}\le o(1) -c(\nu),\\
c(\nu)&=\log 2+\int_0^1\log\left(\frac{1}{2}+\frac{x^{\nu}}{2}\right)\,dx=\int_0^1\log\bigl(1+x^{\nu}\bigr)\,dx.
\end{aligned}
\end{equation}
Clearly, $c(\nu)$ is positive and decreasing as $\nu$ increases. Note that
\begin{align*}
c(1)&=2\log 2-1\approx 0.386, \\
c(\nu)&=\sum_{j\ge 1}\frac{(-1)^{j-1}}{j(\nu j+1)}\rightarrow 0 \mbox{ as } \nu\rightarrow \infty.\\
\end{align*}
Let $\nu<1$. Since $\left(\frac{1}{k}\sum_{j=1}^k u_j^{1/\nu}\right)^{\nu}$ is decreasing in $\nu$, \eqref{easily} 
holds with $c(\nu):=c(1)$. Thus, for all $\nu>0$, the truncated expectation $E_{n,k}$ is
decreasing exponentially with $k\le k_n$. Combining this claim with \eqref{Uk<delta}, we have proved
\begin{Lemma}\label{thm:pu_kv_1}
Under the conditions  (1-2) in Theorem \ref{B}, for $k\le k_n$, we have
\[
\pr\{\bar W_k\le kV_{1}\}\le \exp\bigl (o(k) -c(\max(\nu,1)) k).
\]
\end{Lemma}

It remains to upper-bound $\hat \rho(n,k)$. According to \eqref{P(barW<),hatrho=},
\begin{equation}\label{hatrho=Ex[int]}
\hat\rho(n,k)\le n(k+1)\ex\Biggl[\int_{\bar W_k/k}^{\infty}f\bigl((k+1)v-\bar W_k\bigr)\,\bigl[1-F(v)\bigr]^{n-1}\,dv\Biggr].
\end{equation}
Let us bound the integral. We now assume that the density $f$ satisfies the condition (3): 
\[
\beta=\sup\Bigl\{f(v')/f(v):\,v' \in [v, 2v],\,f(v)>0\Bigr\}<\infty;
\]
following Feller \cite{Fel1}, one can say that the density $f$ is {\it dominatedly varying\/} on the support of the distribution $F$. We will need
the following result, cf. \cite{Fel1}, \cite{Fel2}.
\begin{Lemma}\label{beta(k)<} Introduce
\begin{equation}\label{beta(k)=}
\beta(j)=\sup\left\{\frac{f(x')}{f(x)}: x,\,x'\in (0,b),\,x'\in [x,jx],\,f(x)>0\right\},\quad j>1.
\end{equation}
Under the condition (3), we have $\beta(j)\le \beta j^{\alpha}$, with $\alpha:=\log_2\beta$.
\end{Lemma}
\begin{proof} (For completeness.) First of all, 
\[
\beta(k_1k_2)\le \beta(k_1)\beta(k_2),\quad k_1,\,k_2\ge 2;
\]
thus $\beta(\cdot)$ is a sub-multiplicative function.  Consequently, $\beta(2^m)\le \beta(2)^m=\beta^m$. 
Second, given $k>2$, let $m=m(k)$ be such that $2^{m-1}< k\le 2^{m}$. Since $\beta(\cdot)$ is increasing,
we have 
\begin{align*}
\beta(k)&\le \beta(2^m)\le \beta^m\le \beta^{\frac{\log(2k)}{\log 2}}
=(2k)^{\log_2\beta}=\beta k^{\log_2\beta}.
\end{align*}
\end{proof}
The argument of the density $f$ in \eqref{hatrho=Ex[int]} is sandwiched between $v$ and $kv$. So, by Lemma \ref{beta(k)=}, we obtain
\begin{align*}
\hat\rho(n,k)&\le_b n k^{\log_2\be+1}\ex\Biggl[\int_{\bar W_k/k}^{\infty}f(v)\,\bigl[1-F(v)\bigr]^{n-1}\,dv\Biggr]\\
&=k^{\log_2\be+1}\ex\Bigl[\bigl[1-F(\bar W_k/k)\bigr]^n\Bigr]\\
&=k^{\log_2\be+1}\exp\bigl(-c(\max(\nu,1))k+o(k)\bigr)\\
&=\exp\bigl(-c(\max(\nu,1))k+o(k)\bigr),
\end{align*}
as $\log k=o(k)$ for $k\to\infty$.

Therefore we completed the proof of
\begin{theorem}
	Under the properties (1-3) in Theorem \ref{B}, for $k\le k_n$, 
	\[
	\pr\{K_n=k+1\}\le \exp\bigl (- c(\max(\nu,1)) k+o(k)\bigr).
	\]
\end{theorem}

{{\bf Note.\/}  The conditions (1), (2) relate to a {\bf single\/} point $a$, i.e. they are so mild that there are scores of the classic densities meeting them. The condition (3) is different, as it concerns the
ratio of the density values at the pairs of comparably large/small $x$ and $x'$. For the density $f$ of a cdf F meeting the conditions (1) and (2), the condition (3) is met, for instance, if (i) $f(x)>0$ for $x\in (0,b)$  and (ii) there is $x_1\in (0,b)$ such that $f(x)$ is decreasing on $[x_1,b)$.
}

\subsection{Distributions whose supports are not left bounded}

In this section, we turn to the case when the support of the distribution extends all the way to $-\infty$. Two examples come to mind. One is the normal distribution, with density
\[
f(x)=\frac{1}{\sqrt{2\pi}} e^{-x^2/2}.
\]
It is known, and can be proved via a single integration by parts, that
\[
F(x)=\frac{1+O(|x|^{-1})}{|x|\sqrt{2\pi}}\,e^{-x^2/2},\quad x\to -\infty.
\]
Another example is a positive exponential, with $F(x)=e^x$, for $x<0$, and $F(x)\equiv 1$ for $x\ge 0$. They both are special cases of the distribution $F(x)$ such that, for some $a\in (-\infty,\infty)$, $b>0$, $c>0$, and $\kappa>0$,
\begin{equation}\label{F(-infty)}
F(x)= (c+O(|x|^{-\kappa})) |x|^a \exp(-|x|^b),\quad x\to-\infty,
\end{equation}
which will be the focus of our analysis here. 

We distinguish between two cases depending on the value of $b$. 

\subsubsection{{\bf Case\/} $\bold b\ge \bold 1$.} Recall the notation 
\[
\phi(\bold u)=F\left(\frac{1}{k}\sum_{j=1}^k F^{-1}(u_j)\right),\quad S(\bold u)=\frac{1}{k}\sum_{j=1}^k\log \frac{1}{u_j}.
\]
 We will write $L\gtrsim R$ if $L\ge (1+o(1))R$.

\begin{Lemma}\label{normal}  For $u_k\le \delta_n$, it holds that
$
\phi(\bold u)\gtrsim S(\bold u)^{\min\{0,a/b\}} e^{-S(\bold u)}.
$
\end{Lemma}
\begin{proof} Since $x=F^{-1}(u)$ iff $u=F(x)$, we have: for $u\downarrow 0$, 
\begin{equation}\label{|x|=}
\begin{aligned}
|x|&=\left(\log\frac{c}{u}+ a\log |x|+O(|x|^{-\kappa})\right)^{1/b}\\
&=\left[\log\frac{c}{u}+\frac{a}{b}\log\left(\log\frac{c}{u}+a\log |x|+O(|x|^{-\kappa}\right)\right]^{1/b}\\
&=\left[\log\frac{c}{u}+\frac{a}{b}\log\log\left(\frac{1}{u}\right)+O\left(\frac{\log\log(1/u)}{\log(1/u)}\right)\right]^{1/b}.
\end{aligned}
\end{equation}
As $y^{1/b}$ is concave for $b\ge 1$, we obtain then
\begin{align*}
X&:=\frac{1}{k}\sum_{j=1}^k F^{-1}(u_j)\\
&=-\frac{1}{k}\sum_{j=1}^k\left[\log\frac{c}{u_j}+\frac{a}{b}\log\log\left(\frac{1}{u_j}\right)+O\left(\frac{\log\log(1/u_j)}{\log(1/u_j)}\right)\right]^{1/b}\\
&\ge -\left[\frac{1}{k}\sum_{j=1}^k\left(\log\frac{c}{u_j}+\frac{a}{b}\log\log\left(\frac{1}{u_j}\right)+O\left(\frac{\log\log(1/u_j)}{\log(1/u_j)}\right)\right)\right]^{1/b}\\
&=-\left[\log c+S(\bold u) +o(1)+\frac{a}{bk}\sum_{j=1}^k \log\log\left(\frac{1}{u_j}\right)\right]^{1/b}.
\end{align*}
{\bf (i)\/} If $a\le 0$, then
$
X\ge-\bigl[\log c +S(\bold u)+o(1)\bigr]^{1/b}.
$
Since $X\to-\infty$, we evaluate $F(X)$ using \eqref{F(-infty)}:
\begin{align*}
F(X)&=(c+O(|X|^{-\kappa})) |X|^a \exp(-|X|^b)\\
&\ge (1+o(1)) S(\bold u)^{a/b} e^{-S(\bold u)}.
\end{align*}
{\bf (ii)\/} If $a>0$, then, using concavity of $\log(\cdot)$, we obtain
\begin{align*}
\bold X&\ge -\Biggl[\log c+S(\bold u) +o(1)+\frac{a}{b}\log\Biggl(\frac{1}{k}\sum_{j=1}^k \log\frac{1}{u_j}\Biggr)\Biggr]^{1/b}\\
&=-\left[\log c+S(\bold u) +o(1)+\frac{a}{b}\log S(\bold u)\right]^{1/b}.
\end{align*}
Therefore
\[
F(\bold X)=(c+O(|X|^{-\kappa})) |X|^a \exp(-|X|^b)\ge (1+o(1)) e^{-S(\bold u)}.
\]
\end{proof}
{
\begin{Theorem}\label{caseb>1}   Let $b>1$, and $c\in (3/2,2)$.
For $k\le k_n$, we have
\begin{equation}\label{8/9+exp}
\pr\bigl\{\bar W_k\le kV_1\bigr\}\le_b \left(\frac{8}{9}\right)^{k/4} +\exp\left(-\frac{k\left(\log\frac{n}{k}\right)^{\min\{0,a/b\}}}{ce}\right).
\end{equation}
Consequently, by \eqref{P(Kn=k+1)<rho(n,k)}, 
\[
\pr\bigl\{K_n=k+1\bigr\}\le_b n \left(\frac{8}{9}\right)^{k/4}+n\exp\left(-\frac{k\left(\log\frac{n}{k}\right)^{\min\{0,a/b\}}}{2e}\right).
\]
\end{Theorem}
}
\begin{proof} According to \eqref{key} and the definition of $\phi(\bold U)$, we have: given $S>0$,
\begin{equation}\label{ex=ex1+ex2}
\begin{aligned}
&\qquad\qquad\pr\bigl\{\bar W_k\le kV_1\bigr\}=\ex\bigl[(1-\phi(\bold U))^n\bigr]\\
&=\ex\bigl[1_{\{S(\bold U)>S\}}(1-\phi(\bold U))^n\bigr]
+\ex\bigl[1_{\{S(\bold U)\le S\}}(1-\phi(\bold U))^n\bigr].
\end{aligned}
\end{equation}
By Lemma \ref{A1<S<A2}, if $k\le k_n$ then $S(\bold U)>S:=\log\frac{n}{\a k}$ with probability $\le (\a e/(1-\beta))^{\beta k}=
(8/9)^{k/4}$, if we select $\a=\frac{2}{3e}$ and $\beta=1/4$.
If $S(\bold U)\le S$, then by Lemma \ref{normal} we have 
\begin{align*}
\phi(\bold U)&\ge (1+o(1)) S^{\min\{0,a/b\}}e^{-S}=(1+o(1))\frac{\a k}{n}\left(\log\frac{n}{\a k}\right)^{\min\{0,a/b\}}\\
&=(1+o(1))\frac{k}{(3/2) en}\left(\log\frac{n}{k}\right)^{\min\{0,a/b\}}. 
\end{align*}
In that  case we obtain
\begin{equation}\label{[1-F]^m<}
\bigl(1-\phi(\bold U)\bigr)^n \le \exp\left(-(1+o(1))\frac{k\left(\log\frac{n}{k}\right)^{\min\{0,a/b\}}}{(3/2)e}\right).
\end{equation}
Invoking \eqref{ex=ex1+ex2} yields the finequality \eqref{8/9+exp}. 
\end{proof}

\subsubsection{{\bf Case\/}  $\bold b<\bold 1$.} We consider $k\in [k_1,k_n]$, $k_1=\lceil \log^{\sigma} n\rceil$, $\sigma>1+1/b$.
(The choice of $\sigma$
will become clear shortly.)
Our starting bound for $|X|^b$ is based on the
bottom line in \eqref{|x|=}: for $u_k\le\delta_n$, we have
\begin{equation}\label{X>-}
\begin{aligned}
|X|^b &\le \left[\frac{1}{k}\sum_{j=1}^k\left(\log\frac{1}{u_j}\right)^{1/b}\right]^b\\
&\times\left[1+\left(\frac{a}{b}+O(\log^{-1}(1/u_k))\right)\frac{\log\log (1/u_{k})}{\log (1/u_{k})}\right].
\end{aligned}
\end{equation}
Recall that $u_1,\dots,u_k$ are generic values of the first $k$ order statistics $U_1,\dots, U_k$ for the $(n-1)$ $[0,1]$-uniform,  independent random variables.
To find a usable upper bound for the RHS in \eqref{X>-}, valid for {\it almost all\/} $u_1,\dots,u_k\le \delta_n$,
we need to identify a sufficiently {\it likely\/} upper bound for that RHS when the $k$-tuple $(u_1,\dots,u_k)$ is replaced with
the order statistics $U_1,\dots,U_k$.

To this end, we pick $j_n=\lceil \log^{\sigma_1} n\rceil$, such that $1<\sigma_1<\sigma-1/b$ (a choice made possible by
the condition $\sigma>1+1/b$), and use $b<1$ to bound
\begin{equation}\label{<T_1^b+T_2^b}
\begin{aligned}
&\qquad\qquad\quad \left[\frac{1}{k}\sum_{j=1}^k\left(\log\frac{1}{U_j}\right)^{1/b}\right]^b\le T_1^b+T_2^b,\\
& T_1:=\frac{1}{k}\sum_{j=1}^{j_n}\left(\log\frac{1}{U_j}\right)^{1/b}\!\!,\quad  T_2:=\frac{1}{k}\sum_{j=j_n+1}^{k}\left(\log\frac{1}{U_j}\right)^{1/b}.
\end{aligned}
\end{equation}
{
Pick $\sigma_2\in \bigl(1,\,(\sigma-\sigma_1)b\bigr)$ and let $\sigma_3:=(\sigma-\sigma_1)b-\sigma_2>0$. 
Using 
\begin{align*}
T_1&\le \frac{j_n}{k}\left(\log\frac{1}{U_1}\right)^{1/b},\quad \pr\{U_1\le u\}=1-(1-u)^{n-1}\le nu,\end{align*}
we obtain: 
	\begin{equation}\label{T_1^bsmall}
	\begin{aligned}
	&\pr\Bigl\{T_1^b> \log^{-\sigma_3} n\Bigr\}\le \exp\bigl(-\log^{\sigma_2} n+\log n\bigr)\le \exp\bigl(-0.5\log^{\sigma_2} n\bigr).\\
	\end{aligned}
	\end{equation}
We obviously need $\sigma_2>1$ to overcome $\log n$ term; besides we need all our small probabilities to
be really small, i.e. of order $\exp\bigl(-(\log n)^{1+\Delta}\bigr)$. In contrast, $\sigma_3>0$ is good enough for
the proof.}
Existence of the desired $\sigma_i$ is assured by the starting constraint $\sigma>1+1/b$.
Since $\sigma_2>1$, we see that $T_1^b$ is vanishingly small with super-polynomially high probability, i.e. quite surely.

Turn to $T_2$. Our estimates need to be rather sharp since we expect $T_2^b$ to be q.s. the dominant contribution 
to the RHS in \eqref{<T_1^b+T_2^b}. Here is a key observation.  If $w_j$ are independent negative exponentials, with the same parameter, say $1$,  then, denoting $\bold U=(U_1<\dots< U_{n-1})$, we have 
\begin{equation}\label{Kar}
\bold U\overset{\mathcal D}\equiv \left\{\frac{W_i}{W_{n-1}}\right\}_{i\in [n-1]},\quad W_i:=\sum_{j=1}^iw_j,
\end{equation}
 (Karlin and Taylor \cite{KarTay}, Section 5.4). In particular,
\[
T_1\overset{\mathcal D}\equiv\frac{1}{k}\sum_{j=1}^{j_n}\left(\log\frac{W_{n-1}}{W_j}\right)^{1/b},\quad  T_2\overset{\mathcal D}\equiv\frac{1}{k}\sum_{j=j_n+1}^{k}\left(\log\frac{W_{n-1}}{W_j}\right)^{1/b}.
\]
The relation \eqref{Kar} allows us simply to {\it define\/} $\bold U=\bigl\{W_i/W_{n-1}\bigr\}$. In the case of $T_2$, both $W_{n-1}$ and $W_j$ are  sums of the {\it large\/} numbers of 
independent $w_j$, and this opens the door to the Chernoff-type estimates. 

Here is a general, Chernoff-type, claim. Let $X_1,\dots,X_{\nu}$ be the independent copies of a random variable $X$  such that $f(\la):=\ex\bigl[\exp(\la X)\bigr]$ exists and is two-times
continuously differentiable for $\la\in (-\la_0,\la_0)$, for some $\la_0>0$. Let $S_{\mu}:=\sum_{j=1}^{\mu}X_j$.
Then the distribution of $S_{\mu}$ is exponentially concentrated around $\mu f'(0)$. { More precisely, there exist $\Delta_0\in
(0,1)$ and $\rho>0$ such that, 
\begin{equation}\label{Cherconc}
\pr\bigl\{|S_{\mu}-\mu f'(0)|\ge \mu f'(0)\Delta\bigr\}\le 2e^{-\mu f'(0) \rho\Delta^2},\quad \forall\,\Delta\in (0,\Delta_0).
\end{equation}
}
For our case $X=w$ we have
\[
\ex\bigl[e^{\la w}\bigr]=f(\la):=\frac{1}{1-\la},\quad |\la|<1,
\]
so that $\la_0=1$, $f'(0)=1$. So, by \eqref{Cherconc}, { for some $\Delta_0\in (0,1)$,
\begin{equation}\label{Wapprox}
\pr\bigl\{W_i\in [(1-\Delta)i,(1+\Delta)i]\bigr\}\ge 1 - 2\exp(-i\rho\Delta^2),\quad \forall\,\Delta\in (0,\Delta_0).
\end{equation}
}
In particular, with exponentially high probability, $W_{n-1}\in [n/2,2n]$.
 
Further, using the arithmetic-geometric mean inequality, we write
\begin{align}
&\qquad\qquad\qquad T_2=\frac{1}{k}\sum_{j=j_n+1}^k\left(\log\frac{W_{n-1}}{W_j}\right)^{1/b}\notag\\
&=
\frac{1}{k}\sum_{j=j_n+1}^k\left(\log \frac{W_{n-1}/j}{\frac{1}{j}\sum_{i=1}^j w_i}\right)^{1/b}\le
\frac{1}{k}\sum_{j=j_n+1}^k\left(\log \frac{W_{n-1}/j}{\prod_{i=1}^j w_i^{1/j}}\right)^{1/b}\notag\\
&=\frac{1}{k}\sum_{j=j_n+1}^k\left(\log\frac{W_{n-1}}{j}\right)^{1/b}\left(1-\frac{\log\prod_{i=1}^jw_i^{1/j}}{\log\frac{W_{n-1}}{j}}\right)^{1/b}\notag\\
&\qquad\qquad\qquad\text{using } 1 + x\le e^x\notag\\
&\le\frac{1}{k}\sum_{j=j_n+1}^k\!\left(\log\frac{W_{n-1}}{j}\right)^{1/b}\!\!\cdot\,
\exp\left(\frac{1}{b\log(W_{n-1}/j)}\cdot\frac{\sum_{i=1}^j\log \frac{1}{w_i}}{j}\right).\label{exp(sum)}
\end{align}

So this time we are dealing with the sums of {\it logarithms\/} of the independent exponentials $w_i$. 
Observe that 
\[
f(\la):=\ex\left[\exp\left(\la \log\frac{1}{w}\right)\right]=\int_0^{\infty}e^{-z} z^{-\la}\,dz =\Gamma(1-\la),\quad \la<1.
\]
{So 
\[
f'(\la)=\left(\ex\left[\exp\left(\la \log\frac{1}{w}\right)\right]\right)'_{\la}=\ex\left[\left(\frac{1}{w}\right)^{\la}\cdot \log\frac{1}{w}\right]=
\left(\Gamma(1-\la)\right)',
\]
hence $f'(\la)$ exists, and is continuous, for $\la<1$. Letting $\la\to 0$, we obtain
\[
f'(0)=\ex\left[\log\frac{1}{w}\right]=-\left.\frac{d\Gamma(z)}{dz}\right|_{z=1}=-\gamma,
\]
where $\gamma$ is the Euler constant. Likewise $f''(\la)$ exists and is continuous, for $\la<1$, and (a fun fact)
\[
f''(0)=\ex\left[\log^2\frac{1}{w}\right]=\Gamma''(1)=\gamma^2+\frac{\pi^2}{6}.
\]
See Bateman and Erd\'elyi~\cite{BatErd},  Eq. (7) in Sect. 1.7 and Eq. (10) in Sect.1.9.}
Using \eqref{Cherconc},  we have: for some $\Delta_0\in (0,1)$ and $\rho>0$,
\begin{equation}\label{eventsj}
\pr\left\{\frac{\sum_{i=1}^j\log\frac{1}{w_i}}{j}\ge -\gamma (1-\Delta)\right\}\le 2\exp(-j\rho\Delta^2),\quad \forall\,\Delta\in (0, \Delta_0).
\end{equation}
Since $j>j_n=\lceil \log^{\sigma_1} n\rceil$,  and $\sigma_1>1$, the probability
of the union of the events in \eqref{eventsj} over $j\in [j_n+1,k]$ is of order $\exp\bigl[-\Theta(\log^{\sigma_1} n)\bigr]$. Therefore,
{\it in conjunction with\/} \eqref{Wapprox} for $i=n-1$, we have: with probability $\ge 1-\exp\bigl(-\Theta(\log^{\sigma_1} n)\bigr)$, the bound \eqref{exp(sum)} implies
\begin{multline*}
T_2\le\frac{1+O(\log^{-1}n)}{k}\sum_{j=j_n+1}^k\!\left(\log\frac{W_{n-1}}{j}\right)^{1/b}\\
\le \frac{1+O(\log^{-1}n)}{k}(k-j_n)\left(\log\frac{W_{n-1}}{j_n}\right)^{1/b}\\
=\bigl(1+O(\log^{-1}n)\bigr)\left(\log\frac{n}{j_n}\right)^{1/b},
\end{multline*}
{since $j_n/k\le_b (\log n)^{\sigma_1-\sigma}\ll  (\log n)^{-1/b}$.} Therefore
\begin{equation}\label{T_2^b<}
\pr\left\{T_2^b> \bigl(1+O(\log^{-1}n)\bigr)\cdot \log\frac{n}{j_n}\right\}\le  \exp\bigl(-\Theta(\log^{\sigma_1} n)\bigr).
\end{equation}
Combining \eqref{<T_1^b+T_2^b}, \eqref{T_1^bsmall} and \eqref{T_2^b<}, we obtain: with $U_j=\frac{W_j}{W_{n-1}}$,
\begin{equation}\label{sigma4}
\begin{aligned}
&\pr\Biggl\{\Biggl[\frac{1}{k}\sum_{j=1}^k\left(\log\frac{1}{U_j}\right)^{1/b}\Biggr]^b\ge \bigl(1+O(\log^{-1}n)\bigr)\log\left(\frac{n}{\log^{\sigma_1} n}\right)\Biggr\}\\
&\qquad\qquad\le \exp\bigl(-\Theta(\log^{\sigma_4} n)\bigr),\quad \sigma_4:=\min\{\sigma_1,\sigma_2\}>1.
\end{aligned}
\end{equation}
Finally, using \eqref{Wapprox}, we have $1/U_{k}=W_{n-1}/W_k=\Theta(n/k)\ge \Theta(n^{1/2})$ with probability 
exceeding 
\[
1-\exp(-\Theta(k))-\exp(-\Theta(n))\ge 1-\exp\bigl(-\Theta(\log ^{\sigma}n)\bigr).
\]
So, combining the above estimate with \eqref{X>-}, we arrive at
\begin{equation}\label{P(|X|^b<)}
\begin{aligned}
&\pr\!\left\{\!|X|^b\le \!\left[1+\!\left(\frac{a}{b}+O[\log^{-1}(n/k)]\right)\!\frac{\log\log (n/k)}{\log (n/k)}\right]\log\left(\!\frac{n}{\log ^{\sigma_1}n}\!\right)\!\!\right\}\\
&\qquad\qquad\ge 1-\exp\bigl(-\Theta(\log^{\sigma_5} n)\bigr),\quad \sigma_5:=\min\{\sigma, \sigma_4\}.
\end{aligned}
\end{equation}
For $a\le 0$, the event on the LHS of \eqref{P(|X|^b<)} is contained in the event
\begin{equation}\label{a<0cont}
C_n':=\left\{|X|^b\le \log\frac{n}{\log^{\sigma _1}n} +O\left(\frac{\log\log n}{\log n}\right)\right\}.
\end{equation}
For $a>0$, the containing event is
\begin{equation}\label{a>0cont}
C_n'':=\left\{|X|^b\le \log\frac{n}{(\log n)^{\sigma_1-\frac{2a}{b}+O((\log\log n)^{-1})}}\right\},
\end{equation}
because 
\begin{align*}
(\log n) \cdot \left(\frac{a}{b}+O[\log^{-1}(n/k)]\right)&\!\frac{\log\log (n/k)}{\log (n/k)}\\
&\le \frac{2a}{b}(\log\log n) \left(1+O\left(\frac{\log\log n}{\log n}\right)\right),
\end{align*}
with factor $2$ coming from $k_n=\lceil \a n^{1/2}\rceil$, the largest value of $k$ under consideration. Using 
the formula \eqref{F(-infty)}, i.e.
\[
F(x)= (c+O(|x|^{-\kappa})) |x|^a \exp(-|x|^b),\quad x\to-\infty,
\]
we see that $nF(X)= \Theta\Bigl((\log n)^{\sigma_1+\frac{a}{b}}\Bigr)$ on $C_n'$ (i.e. for $a\le 0$), and
$nF(X)=\Theta\Bigl((\log n)^{\sigma_1-\frac{2a}{b}}\Bigr)$ on $C_n''$ (i.e. for $a>0$). We want $nF(X)\gg \log n$
 on the respective event $C_n$. To ensure this property, we need to have $\sigma_1 >1+\frac{|a|}{b}$ if $a\le 0$,
and $\sigma_1>\frac{2a}{b}+1$ if $a>0$. Recall though that for the desirable $\sigma_i$ to exist, it was necessary and
sufficient to have $\sigma> \sigma_1 + 1/b$. So let us introduce $\sigma(a,b)$, the infimum of admissible $\sigma$, i.e.
\begin{equation}\label{sigma(a,b)=}
\sigma(a,b)=\left\{\begin{aligned}
&1+\frac{|a|+1}{b},&&\text{ if }a\le 0,\\
&1+\frac{2a+1}{b},&&\text{ if }a>0.\end{aligned}\right.
\end{equation}
Given $\sigma>\sigma(a,b)$, we can choose  $\sigma_1^*=\sigma_1^*(a,b)$ as the middle point of its respective admissible range:
\[
\sigma_1^*=\left\{\begin{aligned}
&\frac{1+|a|/b +\sigma -1/b}{2}=1+\frac{|a|}{b}+\frac{\sigma-\sigma(a,b)}{2},&&\text{ if }a\le 0,\\
&\frac{1+2a/b+\sigma-1/b}{2}=1+\frac{2a}{b}+\frac{\sigma-\sigma(a,b)}{2},&&\text{ if } a>0.\end{aligned}\right.
\]
Consequently, both $\sigma_1^*+a/b$ for $a\le 0$ and $\sigma_1^*-2a/b$ for $a>0$ equal $1+0.5(\sigma-\sigma(a,b))$. By \eqref{T_1^bsmall}, $\sigma_2$ can be chosen arbitrarily close to, but strictly less than 
\begin{equation}\label{sigma2*=}
\sigma_2^*=b(\sigma-\sigma_1^*)
=1+\frac{b(\sigma-\sigma(a,b))}{2},
\end{equation}
allowing $\sigma_3$ to be positive. It turns out that $\sigma_2^*<\sigma_1^*$, and consequently 
$\min\{\sigma_j: j\in [5]\setminus \{3\}\}$ can be made arbitrarily close from below to $\sigma_2^*$. Sure enough,
$\sigma_2^*>1$ since $\sigma>\sigma(a,b)$.
We have proved 
\begin{Lemma}\label{b<1} Given $a$, and $b\in (0,1)$, let $\sigma(a,b)$ be defined by \eqref{sigma(a,b)=}.  For $\sigma>
\sigma(a,b)$, let $k_1=\lceil \log ^{\sigma}n\rceil$, $k_n=\lceil \a n^{1/2}\rceil$. Then 
\[
\pr\Biggl\{\forall\,k\in [k_1,k_n]: nF\Biggl(\frac{1}{k}\sum_{j=1}^k F^{-1}(U_j)\!\!\Biggr)\ge \log^{1+c} n\Biggr\}\ge 
1-\exp(-\log^{1+d} n)
\]
if $0<c<\frac{\sigma-\sigma(a,b)}{2}$ and $0<d<\frac{b(\sigma-\sigma(a,b))}{2}$.
\end{Lemma}
The next claim follows directly from Lemma \ref{b<1} and
\[
\pr\bigl\{\bar W_k\le kV_1\bigr\}=\ex \bigl[\bigl(1-\phi(\bold U)\bigr)^n\bigr],\quad \phi(\bold U)=F\Biggl(\frac{1}{k}\sum_{j=1}^k F^{-1}(U_j)\!\!\Biggr).
\]
\begin{Theorem}\label{integral<} Let $b<1$. For all $k\in  [k_1,k_n]$, 
\[
\pr\bigl\{\bar W_k\le kV_1\bigr\} \le \exp\bigl(-\log^{1+d} n\bigr),\quad\forall\, d<\frac{b\,(\sigma-\sigma(a,b))}{2}.
\]
{ Consequently, by \eqref{P(Kn=k+1)<rho(n,k)},
\[
\pr\bigl\{K_n=k+1\bigr\} \le \exp\bigl(-\log^{1+d} n\bigr),\quad\forall\, d<\frac{b\,(\sigma-\sigma(a,b))}{2}.
\]
}
\end{Theorem}

\subsubsection{{\bf Case\/} $\bold Q=(\bold M+\bold M^{\bold T})/{\bold 2}$} Consider the following random matrix model: first randomly generate an $n\times n$ matrix $M$ whose elements are i.i.d.
random variables with the cdf $G(\cdot)$; then define $Q=(M+M^T)/2$. Chen and Peng ~\cite{ChenPeng2015} studied the
case when $G$ is the standard normal distribution. Let us consider the more general case when for some $a$ and  positive
$b$, $c$ and $\kappa$,
\begin{equation}\label{G(-infty)}
G(x)= (c+O(|x|^{-\kappa})) |x|^a \exp(-|x|^b),\quad x\to-\infty.
\end{equation}
We will assume that this asymptotic formula can be differentiated to yield an asymptotic formula for the density $g(x)$.
The diagonal entries of $Q$ have distribution $G$, while  the non-diagonal
entries of $Q$ have the distribution $F(x)= (G\star G)(2x)$. Let us show that $F$ satisfies the condition similar to \eqref{G(-infty)}. {
\begin{Lemma}\label{FOK} Suppose $G$ satisfies the condition \eqref{G(-infty)}.
Suppose that $a>-1$ if $b\le 1$. Then there exist $c'>0$, $\kappa'>0$ such that
\begin{align*}
&F(x)=(c'+O(|x|^{-\kappa'}))|x|^{a'}\exp\Bigl(-2^{\min(1,b)}|x|^b\Bigr),\quad x\to-\infty,\\
&a'=\left\{\begin{aligned}
&2a+\frac{b}{2},&&\text{ if }\, b>1,\\
&a+b-1,&&\text{ if }\,0<b<1,\\
&2a+1,&&\text{ if }\,b=1.
\end{aligned}\right.
\end{align*}
\end{Lemma}}
\noindent {\bf Note.\/} Importantly, thanks to the factor $2^{\min(1,b)}>1$, we have
\begin{equation}\label{frac=infty}
\lim_{x\to -\infty}\frac{G(x)}{F(x)}=\infty.
\end{equation}
Chen, Peng and Zhang ~\cite{ChenPeng2015} demonstrated that, for the diagonal entries and the non-diagonal entries
having respectively the distributions $G$ and $F$,
\[
\pr\bigl\{\bar W_k\le kV_1\bigr\}=\ex\Biggl[\Biggl(1-G\biggl(\frac{1}{k}\sum_{j=1}^k F^{-1}(U_j)\!\!\Biggr)\Biggr)^n\Biggr].
\]
So by \eqref{frac=infty}, 
\begin{equation}\label{weak}
\Biggl(1-G\Biggl(\frac{1}{k}\sum_{j=1}^kF^{-1}(u_j)\Biggr)\Biggr)^n \le 
\Biggl(1-F\Biggl(\frac{1}{k}\sum_{j=1}^kF^{-1}(u_j)\Biggr)\Biggr)^n
\end{equation}
for $\frac{1}{k}\sum_{j=1}^k F^{-1}(u_j)< -\mathcal S$, if $\mathcal S>0$ is sufficiently large. Therefore the argument in the previous section
will apply to this model once we show that $F$ meets the condition \eqref{F(-infty)}.
\begin{proof} { We will write $f(x)\sim g(x)$ if, for some $\omega>0$, $f(x)/g(x)= 1+O(|x|^{-\omega})$ as $x\to-\infty$.
Differentiating the asymptotic formula \eqref{G(-infty)}, we
have 
\[
g(x)=(cb+O(|x|^{-\kappa})) |x|^{a+b-1}\exp(-|x|^b),\quad x\to-\infty.
\]
Now
\begin{equation}\label{M+M<2x}
\begin{aligned}
F(x)&=\pr\{Q_{i,j}\le x\}=\pr\{M_{i,j}+M_{j,i}\le 2x\}\\
&=2\pr\{M_{i,j}+M_{j,i}\le 2x,\,M_{i,j}\le x\}-\pr^2\{M_{i,j}\le x\}.
\end{aligned}
\end{equation}
Here, by \eqref{G(-infty)},
\begin{equation}\label{P^2}
\pr^2\{M_{i,j}\le x\}=(c^2+O(|x|^{-1})) |x|^{2a} \exp(-2|x|^b),\quad x\to-\infty.
\end{equation}
Consider $\pr\{M_{i,j}+M_{j,i}\le 2x,\,M_{i,j}\le x\}$. 

{\bf Case \/} $\bold b>\bold 1$. Picking $\la\in (1,2)$ such that $\la^b>2$, we have: for $x<0$,
\begin{equation}\label{pickla}
\begin{aligned}
&\pr\{M_{i,j}+M_{j,i}\le 2x,\,M_{i,j}\le x\}=\int\limits_{-\infty}^xG(2x-u) g(u)\,du\\
&=\int\limits_{-\infty}^{\la x}G(2x-u)g(u)\,du+\int\limits_{\la x}^xG(2x-u)g(u)\,du=:\int_1+\int_2.
\end{aligned}
\end{equation}

Here, by \eqref{G(-infty)}, for $x\to-\infty$,
\begin{equation}\label{int_1<}
\int_1\le G(\la x)\le 2c |\la x|^a e^{-|\la x|^b},
\end{equation}
and 
\begin{equation}\label{P(Q<x)}
\begin{aligned}
&\int_2\sim c^2b \int\limits_{\la x}^x |2x-u|^a \cdot |u|^{a+b-1}\exp\bigl(-|2x-u|^b-|u|^b\bigr)\,du\\
&\,\,(\psi(x,u):=|2x-u|^b + |u|^b\text{ attains its minimum at }u=x)\\
&\sim c^2b\, |x|^{2a+b-1}\int\limits_{\la x}^x\exp\bigl(-|2x-u|^b-|u|^b\bigr)\,du\\
&\bigl(\text{Taylor-expanding }\psi(x,u)\text{ at } u=x\bigr)\\
&\sim c^2b\, |x|^{2a+b-1}\exp\bigl(-2|x|^b\bigr)\int\limits_{\la x}^{x}\exp\bigl(-b(b-1)|x|^{b-2}(x-u)^2\bigr)\,du\\
&\bigl(\text{extending the integration to }(-\infty,x]\bigr)\\
&\sim c^2b\,|x|^{2a+\frac{b}{2}} \exp\bigl(-2|x|^b\bigr)\sqrt{\frac{\pi}{b(b-1)}}\\
&=\frac{c^2}{2}\sqrt{\frac{b\pi}{b-1}}\,|x|^{2a+\frac{b}{2}}\,\exp\bigl(-2|x|^b\bigr).
\end{aligned}
\end{equation}
Combining the bounds \eqref{G(-infty)}, \eqref{P^2} and \eqref{P(Q<x)}, and recalling that $\la^b >2$, we complete the proof.\\

{\bf Case\/} $\bold b\in (\bold 0,\bold 1]$.
This time we pick $\la>2$ in \eqref{pickla}. 
Substituting $u=\xi x$ in the first line of \eqref{P(Q<x)}, {\it and\/} using $a>-1$, we have: for $x\to-\infty$,
\begin{equation}\label{int_2,bleq1}
\frac{1}{c^2b}\int_2\sim  |x|^{2a+b}\int\limits_{1}^{\la}|2-\xi|^a\,|\xi|^{a+b-1}\exp\Bigl(-|x|^b\bigl(|2-\xi|^b+|\xi|^b\bigr)\Bigr)\,d\xi.
\end{equation}
(It is the factor $|2-\xi|^a$ that dictates the condition $a>-1$.) For $b<1$,  the function $|2-\xi|^b +|\xi|^b$ attains its absolute minimum, which is $2^b$, at two points, $\xi=2$ and
$\xi=0$, but only $\xi=2$ is in $[1,\la]$. Further
\[
|2-\xi|^b +|\xi|^b=2^b+|2-\xi|^b +O(|2-\xi|),\qquad \xi\to 2.
\]
Therefore
\begin{multline*}
\frac{1}{c^2b}\int_2\sim 2^{a+b-1}|x|^{2a+b}\int\limits_{1}^{\la}|2-\xi|^a\exp\Bigl(-|x|^b\bigl(2^b+|2-\xi|^b\bigr)\Bigr)\,d\xi\\
=2^{a+b-1}|x|^{2a+b} e^{-|2x|^b}\int\limits_{1}^{\la}|2-\xi|^ae^{-|x(2-\xi)|^b}\,d\xi\\
=2^{a+b-1}|x|^{2a+b} e^{-|2x|^b}\int_{-1}^{\la-2}\eta^ae^{-|x\eta|^b}\,d\eta\\
\sim2^{a+b}|x|^{a+b-1}e^{-|2x|^b}\int_0^{\infty}z^a e^{-z^b}\,dz\\
= \frac{2^{a+b}}{b} \Gamma\Big(\frac{a+1}{b}\Bigr) |x|^{a+b-1} e^{-|2x|^b}.
\end{multline*}
In combination with \eqref{int_1<}, the constraint $\la>2$ and \eqref{P^2}, this completes the proof  for $b<1$. 

Consider $b=1$. Starting with \eqref{int_2,bleq1}, a similar work shows that, for $a>-1$,  
\[
\frac{1}{c^2b}\int_2\sim |x|^{2a+b}e^{-2|x|}\int_1^2(2-\xi)^a\xi^{a+b-1}\,d\xi.
\]
So, again by \eqref{int_1<}, $\int_1/\int_2$ is of order, roughly, $e^{-(\la-2)|x|}$, and by \eqref{P^2}, $\pr^2\{M_{i,j}\le x\}/\int_2$ is of order $|x|^{-b}$.}
\end{proof}
{
\noindent Lemma \ref{FOK} immediately yields the counterparts of Theorems \ref{caseb>1} and \ref{integral<}.
\begin{Theorem}\label{G,integral<} Let $b\le1$ and $a>-1$. For 
\[
a'=\left\{\begin{aligned}
&a+b-1,&&\text{ if }\,0<b<1,\\
&2a+1,&&\text{ if }\,b=1,\end{aligned}\right.
\quad
\sigma>\sigma(a',b):=\left\{\begin{aligned}
&1+\frac{|a'|+1}{b},&&\text{ if }a'\le 0,\\
&1+\frac{2a'+1}{b},&&\text{ if } a'>0,\end{aligned}\right.
\]
and all $k$ between
$\lfloor \log^{\sigma}n\rfloor$ and $k_n=\lceil \a n^{1/2}\rceil$, ($\a>e\sqrt{2}$), we have
\[
\pr\bigl\{K_n=k+1\bigr\} \le \exp\bigl(-\log^{1+d'} n\bigr),\quad\forall\, d'<\frac{b(\sigma-\sigma(a',b))}{2},\\
\]
\end{Theorem}
\begin{Theorem}\label{G,b>1} Let $b>1$. Then,  for all $k\le k_n$, we have
\begin{align*}
\pr\bigl\{K_n=k+1\bigr\}&\le_b n\left(\frac{8}{9}\right)^{k/4} +n\exp\left(-\frac{k\left(\log\frac{n}{k}\right)^{\min\{0,a'/b\}}}{2e}\right),\\
a'&:=2a+\frac{b}{2}. 
\end{align*}
Consequently $K_n=O_p\bigl((\log n)^{\max(1,1/2-2a/b)}\bigr)$. So for the quasi-normal case $a=-1$, $b=2$ we have
$K_n=O_p\bigl(\log^{3/2} n\bigr)$.
\end{Theorem}
}
{

{\bf Note.\/} The reader certainly noticed that the inequality \eqref{weak} is weaker than the inequality \eqref{frac=infty}.
It may well be possible to lower the powers of $\log n$ in the
likely order of $K_n$ by using \eqref{frac=infty} fully, but the additional technicalities look to be disproportionately
high. 
}

{
\section{Conclusion}\label{sec:conclusion}
Our results, together with \cite{int:Peng-StQP} and \cite{ChenPeng2015}, demonstrate that the optimal solutions of randomly generated StQPs are sparse under very general distribution assumptions. It would be interesting to extend our analysis to portfolio selection problems in which the variance of a portfolio of assets with random returns is minimized (\cite{portfolioselect}) so as to diversify the investment risk. However, it has been observed empirically in the literature (see for instance \cite{int:Opt-Finance}, \cite{GaoLi2013} and \cite{GreenHollifield1992}) that this does not lead to the diversification one would have expected, i.e., the solutions are usually quite sparse, when the empirical covariance matrices are constructed from real data. Our results and/or methodologies may allow us to provide an understanding of the sparsity of portfolio selection problems theoretically. 

The sparsity of solutions holds beyond the randomly generated StQPs (\cite{ChenTeo2013}). See also the references 
to $L_1$ minimization research in the introduction. 
It would be important to identify a broader class of random quadratic
optimization problems with many linear constraints that are likely to have solutions close to an extreme point of the attendant polyhedron. 

It would also be interesting to explore how sparsity can be employed to facilitate the design of algorithms {
that are efficient on average.} One possibility is { to sift through} all possible supports whose sizes are no more than the likely (polylogarithmic) upper bound in our theorems. {Our results indicate that, typically, the running time of even such a primitive  algorithm is of order $\exp(\log^{\alpha} n)$, i.e. well below an exponential order.
In this context, we refer the reader to Ye et al. \cite{YeJuLei2017} who develop a homotopy method for solving a sequence of quadratic programs with slightly varying problem parameters. Their computational experiments demonstrate adaptability of the method to solution sparsity. }

\end{document}